\documentclass[11pt, a4paper, reqno]{amsart}
\usepackage{a4wide}
\usepackage{dsfont}
\usepackage{amsmath, amsthm}
\usepackage{multirow}

\usepackage{color}
\usepackage{ifpdf}
\ifpdf 
  \usepackage[pdftex]{graphicx}
  \DeclareGraphicsExtensions{.pdf,.png,.mps}
  \usepackage{pgf}
  \usepackage{tikz}
\else 
  \usepackage{graphicx}
  \DeclareGraphicsExtensions{.eps,.bmp}
  \usepackage{pgf}
\fi
\usepackage{epic}
\usepackage{wrapfig}

\newtheorem{thm}{Theorem}
\newtheorem{cor}[thm]{Corollary}

\theoremstyle{definition}
\newtheorem{defn}[thm]{Definition}
\newtheorem*{rmk}{Remark}
\newtheorem*{ex}{Example}

\DeclareMathOperator\toht{(31\mhyphen 2)}
\DeclareMathOperator\thto{(2\mhyphen 31)}
\DeclareMathOperator\thot{(2\mhyphen 13)}

\newcommand\set[1]{\left\{#1\right\}}

\newcommand\floor[1]{\left\lfloor#1\right\rfloor}

\def\Z{\mathbb{Z}}
\def\N{\mathbb{N}}
\def\S{\mathfrak{S}}
\def\D{\mathfrak{D}}
\def\A{\mathfrak{A}}
\def\M{\mathfrak{M}}
\def\H{\mathfrak{H}}
\def\I{\mathfrak{I}}
\DeclareMathOperator\EB{EB}

\DeclareMathOperator\NE{NE}
\DeclareMathOperator\ER{ER}

\DeclareMathOperator\exc{exc}
\DeclareMathOperator\cyc{cyc}
\DeclareMathOperator\wex{wex}

\DeclareMathOperator\defi{drop}

\DeclareMathOperator\des{des}

\DeclareMathOperator\dd{dd}
\DeclareMathOperator\da{da}
\DeclareMathOperator\peak{peak}
\DeclareMathOperator\valley{valley}
\DeclareMathOperator\fmax{fmax}

\DeclareMathOperator\cros{cros}
\DeclareMathOperator\nest{nest}
\DeclareMathOperator\stat{stat}
\DeclareMathOperator\cpeak{cpeak}
\DeclareMathOperator\cvalley{cvalley}
\DeclareMathOperator\cdd{cdd}
\DeclareMathOperator\cda{cda}
\DeclareMathOperator\fix{fix}

\mathchardef\mhyphen="2D
\DeclareMathOperator\les{(31\mhyphen 2)}
\DeclareMathOperator\less{(13\mhyphen 2)}
\DeclareMathOperator\res{(2\mhyphen 13)}
\DeclareMathOperator\ress{(2\mhyphen 31)}

\DeclareMathOperator\LES{31\mhyphen 2}

\DeclareMathOperator\RES{2\mhyphen 13}
\DeclareMathOperator\RESS{2\mhyphen 31}

\def\312{\les}
\def\132{\less}
\def\213{\res}
\def\231{\ress}

\title[Certain expansion of Eulerian polynomials via continued fractions] {The symmetric and unimodal expansion of Eulerian polynomials via continued fractions}
\author{Heesung Shin}
\address[Heesung Shin]{Department of Mathematics, Inha University;
 253 Yonhyun-dong, Nam-gu, Incheon, 402-751, South Korea}
\email{shin@inha.ac.kr}
\author{Jiang Zeng}
\address[Jiang Zeng]{Universit\'{e} de Lyon; Universit\'{e} Lyon 1; Institut Camille Jordan; UMR 5208 du CNRS; 43, boulevard du 11 novembre 1918, F-69622 Villeurbanne Cedex, France}
\email{zeng@math.univ-lyon1.fr}

\date{\today}

\begin{document}
\maketitle

\begin{abstract}
This paper was motivated by a conjecture of Br\"{a}nd\'{e}n (European J. Combin. \textbf{29} (2008), no.~2, 514--531)
about the divisibility of the coefficients in an expansion of generalized Eulerian polynomials, which
 implies the symmetric and unimodal property of the Eulerian numbers.
We show that such a formula with the   conjectured property
can be derived  from the
combinatorial theory of continued fractions.  We also discuss an analogous expansion for the corresponding formula for derangements
and prove a $(p,q)$-analogue of the fact that  the (-1)-evaluation of the enumerator polynomials of permutations (resp. derangements) by the number of excedances gives rise to tangent numbers (resp. secant numbers).  The $(p,q)$-analogue unifies and generalizes
our recent results (European J. Combin. \textbf{31} (2010), no.~7, 1689--1705.) and that of Josuat-Verg\`es
(European J. Combin. \textbf{31} (2010), no.~7, 1892--1906).
\end{abstract}


\section{Introduction}
The Eulerian polynomials $A_{n}(t)$ can be defined by
\begin{align*}
\sum_{n=0}^{\infty}A_{n}(t)\frac{x^{n}}{n!}=\frac{1-t}{e^{(t-1)x}-t}.
\end{align*}
Let $\S_n$ be the set of permutations  on $[n]=\set{1,\dots,n}$. For $\sigma=\sigma(1)\dots\sigma(n)\in \S_n$, 
the entry $i \in [n]$  is called an \emph{excedance} (position)  of $\sigma$ if $i < \sigma(i)$ and 
 the number of \emph{excedances} of $\sigma$ is denoted by  $\exc \sigma$.
It is a classical result (cf. \cite{FS70}) that the Eulerian polynomials have the following combinatorial interpretation
\begin{align*}
A_n(t) = \sum_{\sigma\in\S_n}t^{\exc\sigma}=A(n,0)+A(n,1)t+\cdots +A(n,n-1)t^{n-1},
\end{align*}
with $A(n,k)=\#\{\sigma\in \S_{n}: \exc \sigma=k\}$ and the  expansion
\begin{align}\label{eq:peaka}
A_n(t) = \sum_{k=0}^{\lfloor (n-1)/2\rfloor} a_{n,k} t^{k} (1+t)^{n-1-2k},
\end{align}
where $a_{n,k} $ are nonnegative integers with known combinatorial interpretations. Recall that a sequence of real numbers $a_{0}, a_{1}, \ldots, a_{d}$ is said to be symmetric if $a_{i}=a_{d-i}$ for $i=0, \ldots \lfloor d/2\rfloor$ and siad to be unimodal if there exists an index $1\leq j\leq d$ such that $a_{0}\leq a_{1}\leq \cdots \leq a_{j-1}\leq a_{j}\geq a_{j+1}\geq \cdots \geq  a_{d}$.
Note that the expansion \eqref{eq:peaka} enables to derive immediately the \emph{symmetry} and \emph{unimodality} of the Eulerain numbers $\{A(n,k)\}_{0\leq k\leq n-1}$.  Foata and Strehl~\cite{FS76} studied the combinatorial aspect 
of the expansion~\eqref{eq:peaka}  via a group acting on the symmetric groups.  In 2008,
generalizing Foata and Strehl's action,    Br\"{a}nd\'{e}n~\cite{Bra08}  gave a $(p,q)$-refinement of \eqref{eq:peaka}
and  conjectured that the corresponding polynomial coefficient $a_{n,k}(p,q)$  has a factor $(p+q)^k$ for all $0\le k \le \floor{(n-1)/2}$
(see \cite[Conjecture~10.3]{Bra08}).
In this paper  we shall give a new
approach  to such an expansion  with a proof of his conjecture as a bonus (cf. Theorem~\ref{thm:a}).

Next we shall study the derangement counterpart of \eqref{eq:peaka}.
Recall that a permutation $\sigma\in \S_{n}$ is a \emph{derangement}  if it has no  fixed points, i.e.,  $\sigma(i)\neq i$ for all $i\in [n]$.
Let $\D_n$ be the set of \emph{derangements} in $\S_{n}$.  The derangement analogue of the Eulerian polynomials (see  \cite{FS70,Bre90}) is defined by
\begin{align*}
B_n(t) =\sum_{\sigma\in\D_n}t^{\exc\sigma}=B(n,1)t+B(n,2)t^{2}+\cdots +B(n,n-1)t^{n-1},
\end{align*}
with $B(n,k)=\#\{\sigma\in \D_{n}: \exc \sigma=k\}$.  The  generating function for $B_{n}(x)$
  reads as follows (see  \cite{FS70,Bre90,KZ01}):
\begin{align*}
\sum_{n=0}^{\infty}B_{n}(t)\frac{x^{n}}{n!}=\frac{1-t}{e^{tx}-te^{x}}.
\end{align*}
By analytical method, one can show (cf. \cite{Bra06, Zha95}) that  there are non negative integers $b_{n,k}\in \Z$ such that
\begin{align}
B_n(t) = \sum_{k=1}^{\lfloor n/2\rfloor} b_{n,k} t^{k} (1+t)^{n-2k}.\label{eq:peakb}
\end{align}
However, no combinatorial interpretation  for $b_{n,k}$ seems to be known hitherto. A
special case  of our results (cf. Corollary~\ref{cor:pq-euler-secant} and Theorem~\ref{thm: dcycle}) will provide a combinatorial interpretation for the coefficients $b_{n,k}$.

Another  interesting feature of \eqref{eq:peaka} and \eqref{eq:peakb}  is  that they generalize
the well-known relation  between the (-1)-evaluation of $A_n(t)$ and $B_n(t)$ and
 the Euler numbers $E_{n}$  (see \cite{FS70}):
 \begin{align}\label{eq:euler-secant-tangent}
\begin{array}{clcl}
A_{2n}(-1)&=0,\qquad& A_{2n+1}(-1)=&(-1)^{n}E_{2n+1},\\
B_{2n}(-1)&=(-1)^{n}E_{2n},\qquad& B_{2n+1}(-1)=&0,
\end{array}
\end{align}
where the Euler numbers $\{E_{n}\}_{n\geq 0}$ are defined by
 $
 \sum_{n\geq 0}E_{n}\frac{x^{n}}{n!}=\tan x+\sec x.
 $
It follows that  $a_{2n+1,n}=E_{2n+1}$ and $b_{2n,n}=E_{2n}$.
We will prove a
$(p,q)$-analogue of  the  formulae \eqref{eq:euler-secant-tangent} (cf. Theorem~\ref{thm:FH}),  which unifies the recent results in \cite{JV10,SZ10}.

Our main tool is the combinatorial theory  of continued fractions due to
Flajolet~\cite{Fla80} and  bijections due to Fran\c{c}on-Viennot, Foata-Zeilberger and Biane, see \cite{FV79, FZ90, Fla80, Bia93, CSZ97, SZ10} between
permutations and Motzkin paths.
\section{Definitions and main results}\label{sec: main}
Given a permutation $\sigma \in \S_n$, the entry $i\in [n]$ is called a \emph{descent} of $\sigma$ if $i<n$ and $\sigma(i)>\sigma(i+1)$. Also the entry $i\in [n]$ is called a \emph{weak excedance} (resp.  \emph{drop}) of $\sigma$ if  $i \leq  \sigma(i)$ (resp.  $i > \sigma(i)$).
Denote the number of descents (resp. weak excedances, drops) in $\sigma$ by $\des \sigma$  (resp. $\wex \sigma$, $\defi \sigma$).
It is well known \cite{FS70} that the  statistics $\des$  and $\exc$  have the same distribution on~$\S_n$.
Since $\defi \sigma = \exc \sigma^{-1}$ for any $\sigma\in \S_n$,  we see that the Eulerian polynomial
has the following two interpretations
\begin{align}\label{eq:eulerian}
A_n(t)=\sum_{\sigma \in \S_n} t^{\des\sigma} = \sum_{\sigma \in \S_n} t^{\defi\sigma}.
\end{align}
The two above interpretations of Eulerian polynomials give rise then two possible
extensions: one uses linear statistics and the other one uses cyclic statistics. We need some more definitions.

\begin{defn}\label{def:1}
For  $\sigma \in \S_n$,  let   $\sigma(0)=\sigma(n+1)=0$.  Then any entry $\sigma(i)$ ( $i\in [n]$)  can be classified according to one of the four cases:
\begin{itemize}
\item a \emph{peak}  if $\sigma(i-1) < \sigma(i)$ and $\sigma(i) > \sigma(i+1)$;
\item a \emph{valley}  if $\sigma(i-1) > \sigma(i)$ and $\sigma(i) < \sigma(i+1)$;
\item a \emph{double ascent}  if $\sigma(i-1) < \sigma(i)$ and $\sigma(i) < \sigma(i+1)$;
\item a \emph{double descent} if $\sigma(i-1) > \sigma(i)$ and $\sigma(i) > \sigma(i+1)$.
\end{itemize}
Let $\peak^* \sigma$ (resp. $\valley^* \sigma$, $\da^* \sigma$, $\dd^* \sigma$) denote  the number of peaks (resp. valleys, double ascents, double descents) in~$\sigma$. Clearly we have $\peak^* \sigma=\valley^* \sigma+1$.
\end{defn}

For  $\sigma \in \S_n$,
the statistic $\les \sigma$ (resp.  $\less \sigma$ )
 is the number of pairs $(i,j)$ such that $2\leq i<j\leq n$ and $ \sigma(i-1)>\sigma(j)>\sigma(i)$
 (resp.  $ \sigma(i-1)<\sigma(j)<\sigma(i)$).
Similarly,
the statistic $\res\sigma $ (resp. $\ress \sigma $)
is the number of pairs $(i,j)$ such that $1\leq i<j\leq n-1$ and $\sigma(j+1)>\sigma(i)>\sigma(j)$ (resp.
$\sigma(j+1)<\sigma(i)<\sigma(j)$).
Introduce  the generalized  Eulerian polynomial  defined by
\begin{align}\label{eq:dfA}
A_n(p,q,t,u,v,w) := \sum_{\sigma\in \S_n} p^{\res \sigma} q^{\les \sigma}  t^{\des\sigma}u^{\da^*\sigma} v^{\dd^*\sigma}w^{\valley^*\sigma}.
\end{align}
Let $\S_{n,k}$  be  the subset of permutations $\sigma\in \S_{n}$ with exactly $k$ valleys and without double descents.
Define the polynomial
\begin{align} \label{eq:combin}
a_{n,k}(p,q)=\sum_{\sigma\in \S_{n,k}} p^{\res \sigma} q^{\les \sigma}.
\end{align}
Note that the coefficient  $a_{2n+1,n}(p,q)$ is the $(p,q)$-analogue of tangent number, i.e., $E_{2n+1}(p,q)$ in \cite[eq. (14)]{SZ10}.

\begin{thm}\label{thm:a}
We have the expansion formula
\begin{equation}\label{eq:def}
A_n(p,q,t,u,v,w)= \sum_{k=0}^{\floor{(n-1)/2}} a_{n,k}(p,q) (tw)^k (u+vt)^{n-1-2k}.
\end{equation}
Moreover, for all $0\le k \le \floor{(n-1)/2}$, the following divisibility holds
\begin{align}\label{eq:conj}
(p+q)^k  ~|~ a_{n,k}(p,q).
\end{align}
\end{thm}

\begin{rmk}
Br\"and\'en~\cite{Bra08} used the convention
$\sigma(0)=\sigma(n+1)=n+1$ for the definition of  \emph{peak} (resp. \emph{valleys}, \emph{double ascent}, \emph{double descent}). Let $\peak_B, \valley_B, \da_B, \dd_B$ be the corresponding statistics.
To see how  the above theorem implies his conjecture we
define the \emph{reverse} and \emph{complement}  transformations on the permutations $\sigma\in \S_{n}$ by
\begin{align*}
\sigma \mapsto \sigma^{r}&=\sigma(n)\sigma(n-1)\ldots \sigma(1),\\
\sigma \mapsto \sigma^{c}&=(n+1-\sigma(1))(n+1-\sigma({2}))\ldots (n+1-\sigma(n)).
\end{align*}
Clearly the reverse-complement transformation
$\sigma\mapsto \sigma^{rc}:=(\sigma^{r})^{c}$ satisfies
\begin{multline}
(\des, \peak^*, \valley^*, \da^*, \dd^*, \res, \312) \sigma\\
=  (\des, \valley_B, \peak_B, \da_B, \dd_B, \132, \231)\sigma^{rc}.
\end{multline}
Thus, setting $t=u=v=1$ in Theorem~\ref{thm:a},
we have
\begin{equation}
\sum_{\sigma\in \S_n} p^{\132 \sigma} q^{\231 \sigma} w^{\peak_B \sigma}
= \sum_{k=0}^{\floor{(n-1)/2}} a_{n,k}(p,q) 2^{n-1-2k} w^k.
\label{eq:case}
\end{equation}
Comparing two coefficients of $w^k$ in both sides of \eqref{eq:case}, we recover
Br\"{a}nd\'{e}n's  result~\cite[(5.1)]{Bra08}: 
\begin{align}\label{eq:branden}
a_{n,k}(p,q)=2^{-n+1+2k}\sum_{{\pi\in \S_n\atop  \peak_B\pi=k}} p^{(13\mhyphen2)\pi} q^{(2\mhyphen31)\pi}.
\end{align}
Moreover  \eqref{eq:conj} confirms  his conjecture~\cite[Conjecture 10.3]{Bra08}.
\end{rmk}

\begin{defn}
For  $\sigma\in \S_n$,  a  value $x=\sigma(i)$ ($i\in [n]$) is called
\begin{itemize}
\item a \emph{cyclic peak}  if $i=\sigma^{-1}(x)<x$ and $x >\sigma(x)$;
\item a \emph{cyclic valley}  if $i=\sigma^{-1}(x)>x$ and $x<\sigma(x)$;
\item a \emph{double excedance} if $i=\sigma^{-1}(x)<x$ and $x<\sigma(x)$; 
\item a  \emph{double drop} if $i=\sigma^{-1}(x)>x$ and $x>\sigma(x)$;
\item a \emph{fixed point}  if $x=\sigma(x)$.
\end{itemize}
Let $\cpeak \sigma$ (resp. $\cvalley \sigma$, $\cda \sigma$, $\cdd \sigma$, $\fix \sigma$) be the number of cyclic peaks (resp. valleys, double excedances, double drops, fixed points) in $\sigma$.
\end{defn}

For a permutation $\sigma\in \S_{n}$ the crossing and nesting numbers are defined by
\begin{align}
 \cros\sigma&= \# \set{(i,j)\in [n]\times[n]: (i<j\le\sigma(i)<\sigma(j))\vee(i>j>\sigma(i)>\sigma(j))},\\
\nest\sigma &= \# \set{(i,j)\in [n]\times[n]: (i<j\le\sigma(j)<\sigma(i))\vee(i>j>\sigma(j)>\sigma(i))}.
\end{align}
As  in \cite{Cor07}, we can  illustrate  these statistics  by a permutation diagram.  Given $\sigma\in \S_n$, we  put the numbers from $1$ to $n$
on a line and draw an edge from $i$ to $\sigma(i)$ above the line if $i$ is a weak excedance,  and below the line otherwise.
For example, the permutation diagram of $\sigma=3762154=\left({1 \atop 3}{2 \atop 7}{3 \atop 6}{4 \atop 2}{5 \atop 1}{6 \atop 5}{7 \atop 4}\right)$ is as follows:
$$
\centering
\begin{pgfpicture}{8.00mm}{61.14mm}{92.00mm}{94.43mm}
\pgfsetxvec{\pgfpoint{1.00mm}{0mm}}
\pgfsetyvec{\pgfpoint{0mm}{1.00mm}}
\color[rgb]{0,0,0}\pgfsetlinewidth{0.30mm}\pgfsetdash{}{0mm}
\pgfmoveto{\pgfxy(10.00,80.00)}\pgflineto{\pgfxy(90.00,80.00)}\pgfstroke
\pgfsetlinewidth{0.15mm}\pgfmoveto{\pgfxy(20.00,80.00)}\pgfcurveto{\pgfxy(21.04,82.19)}{\pgfxy(22.81,83.96)}{\pgfxy(25.00,85.00)}\pgfcurveto{\pgfxy(30.53,87.63)}{\pgfxy(37.15,85.42)}{\pgfxy(40.00,80.00)}\pgfstroke
\pgfmoveto{\pgfxy(40.00,80.00)}\pgfcurveto{\pgfxy(41.29,82.00)}{\pgfxy(43.00,83.71)}{\pgfxy(45.00,85.00)}\pgfcurveto{\pgfxy(53.32,90.35)}{\pgfxy(64.38,88.14)}{\pgfxy(70.00,80.00)}\pgfstroke
\pgfmoveto{\pgfxy(70.00,80.00)}\pgfcurveto{\pgfxy(70.00,77.24)}{\pgfxy(67.76,75.00)}{\pgfxy(65.00,75.00)}\pgfcurveto{\pgfxy(62.24,75.00)}{\pgfxy(60.00,77.24)}{\pgfxy(60.00,80.00)}\pgfstroke
\pgfmoveto{\pgfxy(50.00,80.00)}\pgfcurveto{\pgfxy(47.15,74.58)}{\pgfxy(40.53,72.37)}{\pgfxy(35.00,75.00)}\pgfcurveto{\pgfxy(32.81,76.04)}{\pgfxy(31.04,77.81)}{\pgfxy(30.00,80.00)}\pgfstroke
\pgfmoveto{\pgfxy(60.00,80.00)}\pgfcurveto{\pgfxy(51.62,69.12)}{\pgfxy(36.09,66.90)}{\pgfxy(25.00,75.00)}\pgfcurveto{\pgfxy(23.09,76.40)}{\pgfxy(21.40,78.09)}{\pgfxy(20.00,80.00)}\pgfstroke
\pgfputat{\pgfxy(40.00,76.00)}{\pgfbox[bottom,left]{\fontsize{11.38}{13.66}\selectfont \makebox[0pt]{$3$}}}
\pgfputat{\pgfxy(20.00,76.00)}{\pgfbox[bottom,left]{\fontsize{11.38}{13.66}\selectfont \makebox[0pt]{$1$}}}
\pgfputat{\pgfxy(30.00,76.00)}{\pgfbox[bottom,left]{\fontsize{11.38}{13.66}\selectfont \makebox[0pt]{$2$}}}
\pgfputat{\pgfxy(50.00,76.00)}{\pgfbox[bottom,left]{\fontsize{11.38}{13.66}\selectfont \makebox[0pt]{$4$}}}
\pgfputat{\pgfxy(60.00,76.00)}{\pgfbox[bottom,left]{\fontsize{11.38}{13.66}\selectfont \makebox[0pt]{$5$}}}
\pgfcircle[fill]{\pgfxy(40.00,80.00)}{1.00mm}
\pgfsetlinewidth{0.30mm}\pgfcircle[stroke]{\pgfxy(40.00,80.00)}{1.00mm}
\pgfputat{\pgfxy(50.00,64.00)}{\pgfbox[bottom,left]{\fontsize{11.38}{13.66}\selectfont \makebox[0pt]{$\sigma=3762154$}}}
\pgfmoveto{\pgfxy(29.50,86.50)}\pgflineto{\pgfxy(31.00,86.00)}\pgflineto{\pgfxy(29.50,85.50)}\pgfclosepath\pgffill
\pgfmoveto{\pgfxy(29.50,86.50)}\pgflineto{\pgfxy(31.00,86.00)}\pgflineto{\pgfxy(29.50,85.50)}\pgfclosepath\pgfstroke
\pgfmoveto{\pgfxy(54.17,88.28)}\pgflineto{\pgfxy(55.67,87.78)}\pgflineto{\pgfxy(54.17,87.28)}\pgfclosepath\pgffill
\pgfmoveto{\pgfxy(54.17,88.28)}\pgflineto{\pgfxy(55.67,87.78)}\pgflineto{\pgfxy(54.17,87.28)}\pgfclosepath\pgfstroke
\pgfmoveto{\pgfxy(40.50,70.83)}\pgflineto{\pgfxy(39.00,70.33)}\pgflineto{\pgfxy(40.50,69.83)}\pgfclosepath\pgffill
\pgfmoveto{\pgfxy(40.50,70.83)}\pgflineto{\pgfxy(39.00,70.33)}\pgflineto{\pgfxy(40.50,69.83)}\pgfclosepath\pgfstroke
\pgfmoveto{\pgfxy(65.50,72.61)}\pgflineto{\pgfxy(64.00,72.11)}\pgflineto{\pgfxy(65.50,71.61)}\pgfclosepath\pgffill
\pgfmoveto{\pgfxy(65.50,72.61)}\pgflineto{\pgfxy(64.00,72.11)}\pgflineto{\pgfxy(65.50,71.61)}\pgfclosepath\pgfstroke
\pgfmoveto{\pgfxy(40.50,74.50)}\pgflineto{\pgfxy(39.00,74.00)}\pgflineto{\pgfxy(40.50,73.50)}\pgfclosepath\pgffill
\pgfmoveto{\pgfxy(40.50,74.50)}\pgflineto{\pgfxy(39.00,74.00)}\pgflineto{\pgfxy(40.50,73.50)}\pgfclosepath\pgfstroke
\pgfmoveto{\pgfxy(60.00,80.50)}\pgflineto{\pgfxy(60.00,80.50)}\pgfstroke
\pgfputat{\pgfxy(70.00,76.00)}{\pgfbox[bottom,left]{\fontsize{11.38}{13.66}\selectfont \makebox[0pt]{$6$}}}
\pgfsetlinewidth{0.15mm}\pgfmoveto{\pgfxy(30.00,80.00)}\pgfcurveto{\pgfxy(31.46,81.86)}{\pgfxy(33.14,83.54)}{\pgfxy(35.00,85.00)}\pgfcurveto{\pgfxy(48.86,95.86)}{\pgfxy(68.86,93.63)}{\pgfxy(80.00,80.00)}\pgfstroke
\pgfcircle[fill]{\pgfxy(34.91,84.97)}{1.00mm}
\pgfsetlinewidth{0.30mm}\pgfcircle[stroke]{\pgfxy(34.91,84.97)}{1.00mm}
\pgfcircle[fill]{\pgfxy(54.98,74.98)}{1.00mm}
\pgfcircle[stroke]{\pgfxy(54.98,74.98)}{1.00mm}
\pgfmoveto{\pgfxy(54.38,92.43)}\pgflineto{\pgfxy(55.88,91.93)}\pgflineto{\pgfxy(54.38,91.43)}\pgfclosepath\pgffill
\pgfmoveto{\pgfxy(54.38,92.43)}\pgflineto{\pgfxy(55.88,91.93)}\pgflineto{\pgfxy(54.38,91.43)}\pgfclosepath\pgfstroke
\pgfputat{\pgfxy(80.00,76.00)}{\pgfbox[bottom,left]{\fontsize{11.38}{13.66}\selectfont \makebox[0pt]{$7$}}}
\pgfsetlinewidth{0.15mm}\pgfmoveto{\pgfxy(80.00,80.00)}\pgfcurveto{\pgfxy(78.71,78.00)}{\pgfxy(77.00,76.29)}{\pgfxy(75.00,75.00)}\pgfcurveto{\pgfxy(66.68,69.65)}{\pgfxy(55.62,71.86)}{\pgfxy(50.00,80.00)}\pgfstroke
\pgfmoveto{\pgfxy(65.50,75.61)}\pgflineto{\pgfxy(64.00,75.11)}\pgflineto{\pgfxy(65.50,74.61)}\pgfclosepath\pgffill
\pgfsetlinewidth{0.30mm}\pgfmoveto{\pgfxy(65.50,75.61)}\pgflineto{\pgfxy(64.00,75.11)}\pgflineto{\pgfxy(65.50,74.61)}\pgfclosepath\pgfstroke
\end{pgfpicture}%
$$
Clearly, there are three crossings $(1<2<\sigma(1)<\sigma(2))$, $(1<3\le \sigma(1)<\sigma(3))$, $(7>4>\sigma(7)>\sigma(4))$ and three nestings $(2<3<\sigma(3)<\sigma(2))$, $(5>4>\sigma(4)>\sigma(5))$, $(7>6>\sigma(6)>\sigma(7))$, thus $\cros\sigma=\nest\sigma=3$.


\begin{defn}
Given a permutation $\sigma \in \S_n$, let $\sigma(0)=0$ and $\sigma(n+1)=n+1$.
The corresponding number of peaks, valleys, double ascents, and double descents of permutation $\sigma \in \S_n$ is  denoted by $\peak \sigma$, $\valley \sigma$, $\da \sigma$, and $\dd \sigma$, respectively.
Moreover, a  double ascent  $\sigma(i)$ of  $\sigma$ ($i \in [n]$) is said to be a \emph{foremaximum} if
$\sigma(i)$ is  a \emph{left-to-right maximum} of  $\sigma$, i.e., $\sigma(j)<\sigma(i)$ for all $1\leq j<i$.
Denote the number of foremaxima of $\sigma$ by $\fmax\sigma$.
\end{defn}

For instance, $\da(42157368)=3$, but $\da^*(42157368)=2$ and  $\fmax(42157368)=2$.  Note that  by the above definition we have $\peak\sigma=\valley\sigma$ for any $\sigma\in \S_n$.

\begin{thm}\label{thm3}
There is a bijection $\Phi$ on $\S_n$ such that  for all $\sigma\in \S_n$ we have
\begin{align*}
&(\nest , \cros, \defi, \cda , \cdd, \cvalley, \fix)\sigma\\
& \qquad=(\RESS, \LES, \des, \da - \fmax, \dd, \valley, \fmax)\Phi(\sigma).
\end{align*}
\end{thm}

Since $\wex=n-\defi$, we derive immediately from Theorems~\ref{thm:a} and \ref{thm3}  the following $(p,q)$-analogue  of \eqref{eq:peaka}.
\begin{cor}\label{cor:pq-euler-tangent} We have
 \begin{align}
  \sum_{\sigma\in \S_n} p^{\nest\sigma} q^{\cros \sigma} t^{\wex\sigma}
  = \sum_{k=0}^{\floor{(n-1)/2}} a_{n,k}(p,q) t^{k+1}(1+t)^{n-1-2k}.
\end{align}
\end{cor}

Consider the common enumerator polynomial
\begin{align}
B_n(p,q,t,u,v,w,y) :&=
\sum_{\sigma\in \S_n} p^{\nest \sigma} q^{\cros \sigma} t^{\defi\sigma} u^{\cda\sigma} v^{\cdd\sigma} w^{\cvalley\sigma}y^{\fix\sigma}\label{eq:defB} \\
&= \sum_{\sigma\in \S_n} p^{\ress \sigma} q^{\les \sigma} t^{\des\sigma} u^{\da\sigma - \fmax\sigma} v^{\dd\sigma}w^{\valley\sigma}y^{\fmax\sigma}.
 \label{eq:q-Euler_D}
\end{align}

When  $u=0$ and $t=v=1$ we  can write
\begin{align}
B_n(p,q,1,0,1,w,y) :=\sum_{k,j\geq 0}b_{n,k,j}(p,q) w^k y^d,
\end{align}
where $b_{n,k,j}(p,q)$ is a polynomial in $p$ and $q$ with non negative integral coefficients.
In order to give a combinatorial interpretation  for $b_{n,k,j}(p,q)$ we introduce some more definitions.
Let   $\S_{n,k,j}$  denote  the subset of
all the permutations $\sigma\in \S_{n}$ with exactly $k$ cyclic valleys,
$d$ fixed points, and without double excedance,  and let
 $\S_{n,k,j}^*$ denote   the subset of
 all permutations $\sigma\in \S_{n}$ with exactly $k$ valleys and $j$ double ascents, which are
 all  foremaxima.
We derive from Theorem~\ref{thm3} the following result.
\begin{cor}  \label{co:key7} We have
\begin{align}\label{eq:combin_D}
b_{n,k,d}(p,q)=\sum_{\sigma\in \S_{n,k,j}} p^{\nest \sigma} q^{\cros \sigma}=\sum_{\sigma\in \S_{n,k,j}^*} p^{\231 \sigma} q^{\312 \sigma},
\end{align}
In particular, when $j=0$, we obtain
\begin{align}
b_{n,k,0}(p,q) =\sum_{\sigma\in \D_{n,k,0}} p^{\nest \sigma} q^{\cros \sigma} =
 \sum_{\sigma\in \D^*_{n,k,0}} p^{\231 \sigma} q^{\312 \sigma}.
\end{align}
\end{cor}

Recall (see \cite{SZ10}) that a {\em coderangement} is a permutation without  foremaximum.
Let $\D^*_n$ be the subset of $\S_n$  consisting of coderangements, that is,
$\D^*_n = \set{\sigma\in\S_n: \fmax\sigma=0}.$
For example, we have
$ \D^*_4 = \set{2143, 3142, 3241, 4123, 4132, 4213, 4231, 4312, 4321}$.
Thus,  $\D_{n,k,0}$ is the subset of derangements $\sigma\in \D_{n}$ with exactly $k$ cyclic valleys, and without
double excedance,  and $\D^*_{n,k,0}$ is the subset of coderangements $\sigma\in \D^*_{n}$ with exactly $k$ valleys and without double ascents.
The following is our main result about the polynomial $B_n(p,q,t,u,v,w,y)$.
\begin{thm}\label{thm:b}
We have the following decomposition
\begin{equation}\label{eq:def_D}
B_n(p,q,t,u,v,w,y)=  \sum_{j=0}^{n} y^{j}\sum_{k=0}^{\floor{(n-j)/2}} b_{n,k,j}(p,q) (tw)^k (qu+tv)^{n-j-2k}.
\end{equation}
\end{thm}

%

For  any  $\sigma \in \D_n$,
since $\exc \sigma =n-\defi \sigma$, setting $y=0$ in  Theorems~\ref{thm:b}
 we obtain immediately the following $(p,q)$-analogue of \eqref{eq:peakb}.
\begin{cor} \label{cor:pq-euler-secant}
We have
\begin{align}\label{eq:FSbis}
\sum_{\sigma\in \D_{n}}p^{\nest \sigma} q^{\cros \sigma}t^{\exc \sigma}=\sum_{k=0}^{\lfloor n/2\rfloor}b_{n,k,0}(p,q)t^{k}(1+qt)^{n-2k}.
\end{align}
\end{cor}

Note that the coefficient  $b_{2n,n,0}(p,q)$ is the $(p,q)$-analogue of
\emph{secant number}, namely, $E_{2n}(p,q)$ in \cite[eq. (15)]{SZ10}.
A permutation $\sigma\in \S_n$ is said to be a
{\em (falling) alternating permutation} if $\sigma(1)>\sigma(2)$, $\sigma(2)<\sigma(3)$, $\sigma(3)>\sigma(4)$, etc.
Let $\A_n$ be the set of (falling) alternating permutations on $[n]$.
 It is  a folklore  result that the cardinality of
$\A_n$ is the Euler number $E_n$. We derive from Corollaries~\ref{cor:pq-euler-tangent}, \ref{co:key7} and \ref{cor:pq-euler-secant} the following
$(p,q)$-analogue of \eqref{eq:euler-secant-tangent}.
\begin{thm}\label{thm:FH}
For $n\geq 1$, we have
\begin{align*}
\sum_{\sigma\in \S_n}(-1)^{\wex \sigma } p^{\nest \sigma} q^{\cros\sigma}
&=
\begin{cases}
0&\textrm{if $n$ is even},\\
(-1)^{\frac{n+1}{2}} \sum_{\sigma\in \A_{n}}  p^{\thot \sigma} q^{ \toht\sigma}&\textrm{if $n$ is odd};
\end{cases}\\
\intertext{and}
\sum_{\sigma\in \D_n}(-1/q)^{\exc \sigma}p^{\nest \sigma}q^{\cros\sigma}
&=
\begin{cases}
(-1/q)^{\frac{n}{2}} \sum_{\sigma\in \A_{n}}  p^{\thto \sigma}  q^{ \toht\sigma}&\textrm{if $n$ is even},\\
0 &\textrm{if $n$ is odd}.
\end{cases}
\end{align*}
\end{thm}

\begin{rmk}
The $p=1$ case of Theorem~\ref{thm:FH}
  was proved  by
 Josuat-Verg\`es~\cite[Eq. (5)-(6)]{JV10}.  Since
 the \emph{inversion} number of  any $\sigma\in \S_{n}$ can be written as
   $\defi \sigma+\cros \sigma+2\nest\sigma$ (cf. \cite{SZ10}),
 the  $p=q^{2}$ case of  Theorem~\ref{thm:FH}
 was used  by
  Shin and Zeng~\cite[Eq. (12)-(13)]{SZ10} to derive a $q$-analogue of
\eqref{eq:euler-secant-tangent} using the \emph{inversion number} $q$-analogue of Eulerian polynomials.
\end{rmk}

For any permutation $\sigma\in \S_n$, we denote by $\cyc\sigma$
the number of its cycles.
 Define
\begin{align}\label{def:d}
c_{n,k}(\beta)=\sum_{\sigma\in \D_{n}(k)} \beta^{\cyc \sigma},
\end{align}
where $\D_{n}(k)$ is the subset of derangements in $\D_{n}$
with exactly $k$ cyclic valleys and without cyclic double descents.
Clearly $c_{n,0}(\beta)=0$ for all $n\ge1$.
The following result gives another generalization of the expansion \eqref{eq:peakb}.

\begin{thm}\label{thm: dcycle}
We have
\begin{align}\label{eq:d}
\sum_{\sigma\in \D_{n}} \beta^{\cyc \sigma} t^{\exc \sigma}=\sum_{k=0}^{\floor{n/2}} c_{n,k}(\beta)t^{k}(1+t)^{n-2k}.
\end{align}
\end{thm}

The rest of this paper is organized as follows.
In Sections~\ref{sec:FV}--\ref{sec: dcycle} we shall prove Theorem~\ref{thm:a},  Theorem~\ref{thm3},
Theorem~\ref{thm:b}  and
Theorem~\ref{thm: dcycle}, respectively. In Section~\ref{sec:new} we give
two  variations   of the unimodal and symmetric expansion of Eulerian polynomials and one for  its derangement analogue. Finally we conclude with an open problem related to the  descent polynomial of involutions.


\section{Proof of Theorem~\ref{thm:a}} \label{sec:FV}

A Motzkin path of length $n$ is a sequence of points $(s_{0}, \ldots, s_{n})$ in the plan $\N\times \N$ such that
$s_{0}=(0,0)$, $s_{i}-s_{i-1}=(1,0), (1,\pm 1)$ and $s_{n}=(n,0)$.
Denote by $\M_{n}$ the set of Motzkin paths of length $n\geq 1$.
We shall call a step $(s_{i-1}, s_{i})$ East, North-East,  South-East, respectively, if
$s_{i}-s_{i-1}=(1,0)$,  $s_{i}-s_{i-1}=(1,1)$, $s_{i}-s_{i-1}=(1,-1)$.
The \emph{height} of the step $(s_{i-1}, s_{i})$  is the ordinate of $s_{i-1}$.  Given a Motzkin path $\gamma$,  if  we weight each East (resp. North-East,  South-East)  step of height $i$ by $a_{i}$ (resp. $b_{i}$ and $c_{i}$) and define the weight
of  $\gamma$ by the product of its step weights,  denoted by  $w(\gamma)$,
then
\begin{align}\label{eq:fla}
1+\sum_{n\geq 1}\sum_{\gamma\in \M_{n}}w(\gamma) x^{n}=\dfrac{1}
{1-b_{0}x-\dfrac{a_{0}c_{1}x^2}
{1-b_{1}x-\dfrac{ a_{1}c_{2}x^2}
{\cdots}}}.
\end{align}
It is convenient to use two kinds of horizontal steps, say, \emph{blue} and \emph{red}.
Let $\M_{n}'$ be the set of \emph{colored Motzkin paths} of length $n$.
A \emph{Laguerre history} of length $n$ is a couple
$(\gamma, (p_{1}, \ldots, p_{n}))$, where $\gamma$ is a colored Motzkin path of length $n$ and $(p_{1}, \ldots, p_{n})$ is a sequence such that
$0\leq p_{i}\leq v(s_{i-1}, s_{i})$, where $v(s_{i-1}, s_{i})=k$ if $s_{i-1}=(i-1,k)$.
Denote by $\H_{n}$ the set of Laguerre histories  of length $n$.

Let $\sigma\in \S_{n}$, the refinements of three generalized patterns are defined by
\begin{align*}
\les_{k}\sigma &= \# \set{i:i+1<j \text{ and } \sigma(i+1)<\sigma(j)=k<\sigma(i)},\\
\ress_{k}\sigma &= \# \set{i:j<i-1 \text{ and } \sigma(i)<\sigma(j)=k<\sigma(i-1)},\\
\res_{k}\sigma &= \# \set{i:j<i-1 \text{ and } \sigma(i-1)<\sigma(j)=k<\sigma(i)}.
\end{align*}
We clearly have $\les\sigma=\sum_{k=1}^{n}\les_{k}\sigma$, $\ress\sigma=\sum_{k=1}^{n}\ress_{k}\sigma$, and $\res\sigma=\sum_{k=1}^{n}\res_{k}\sigma$.
Two numbers $l_{k}=\les_{k}\sigma$ and $r_{k}=\ress_{k}\sigma$ are called the {\em left embracing numbers} and {\em right embracing numbers} of $k \in [n]$ in $\sigma$.

We need a modified version of Fran\c{c}on-Viennot's bijection $\Psi_{FV} : \S_{n} \to \H_{n-1}$ defined by the following:
for  $\sigma\in \S_{n}$ we construct the Laguerre history
$(s_{0}, \ldots, s_{n-1}, p_{1},\ldots, p_{n-1})$, where $s_{0}=(0,0)$ and the step
$(s_{i-1}, s_{i})$ is  North-East, South-East, East blue and East red
if $i$ is a valley, peak, double ascent, or double descent, respectively; while $p_{i}= \res_{i}\sigma$ for $i=1,\ldots, n-1$.
If $h_{i}$ is the height of $(s_{i-1}, s_{i})$, i.e., $s_{i-1}=(i-1, h_{i})$,  then $\res_{i}\sigma+\les_{i}\sigma=h_{i}$. Since  $\sigma(0)=\sigma(n+1)=0$, so $n$ must be a peak and $\valley\sigma=\peak \sigma-1$. Thus $(s_{0}, \ldots, s_{n-1}, p_{1},\ldots, p_{n-1})$ is a Laguerre history of length $n-1$ and
$$
w(\sigma)=t^{\ER \gamma+\NE \gamma}u^{\EB\gamma}v^{\ER\gamma}w^{\NE \gamma}\prod_{i=1}^{n-1}p^{p_{i}}q^{h_{i}-p_{i}},
$$
where $\NE \gamma$, $\EB \gamma$, and $\ER\gamma$ are the number of North-East steps, East blue steps, and East red steps of $\gamma$.
Therefore,
\begin{align}\label{eq:gf}
A_n(p,q,t,u,v,w)&=\sum_{\gamma \in \M_{n-1}'} t^{\ER \gamma+\NE \gamma}u^{\EB\gamma}v^{\ER\gamma}w^{\NE \gamma} \prod_{i=1}^{n-1}[h_{i}+1]_{p,q},
\end{align}
where $[n]_{p,q}=(p^n-q^n)/(p-q)$.
Given a Motzkin path $\gamma$, weight each step at height $k$ by
\begin{align}\label{eq:weight}
a_k:=tw[k+1]_{p,q},\quad b_k := (u+tv)[k+1]_{p,q},\quad c_k := [k+1]_{p,q}
\end{align}
if the step is North-East, East, and South-East, respectively,
and the weight of $\gamma$ is defined to be the product of  the step weights. Then the last sum amounts to sum over all the Motzkin paths
of length $n-1$ with respect to \eqref{eq:weight}, i.e.,
\begin{align}\label{eq:gf2}
 A_n(p,q,t,u,v,w)&
= \sum_{\gamma \in \M_{n-1}} w(\gamma).
\end{align}

It follows from \eqref{eq:fla} that  the equation \eqref{eq:gf2} is equivalent  to
the following continued fraction expansion
\begin{equation}\label{eq:A}
\sum_{n\ge1} A_n(p,q,t,u,v,w)x^{n-1}=
\dfrac{1}
{1-(u+tv)[1]_{p,q}x-\dfrac{[1]_{p,q}[2]_{p,q}twx^2}
{1-(u+tv)[2]_{p,q}x-\dfrac{[2]_{p,q}[3]_{p,q}twx^2}
{\cdots}}}.
\end{equation}
For $0\leq k\leq {\floor{(n-1)/2}} $, let $a_{n,k}(p,q,t,u,v)$ be the coefficient of $w^k$ in $A_n(p,q,t,u,v,w)$, i.e.,
\begin{equation}
A_n(p,q,t,u,v,w) = \sum_{k=0}^{\floor{(n-1)/2}} a_{n,k}(p,q,t,u,v) w^k.
\label{eq:interpret1}
\end{equation}
Substituting $x \leftarrow \frac{x}{(u+tv)}$ and $w \leftarrow \frac{w(u+tv)^2}{t}$ in  \eqref{eq:A},
we obtain
\begin{equation}\label{eq:b}
\sum_{n\ge1}\sum_{k=0}^{\floor{(n-1)/2}} \frac{a_{n,k}(p,q,t,u,v)}{t^k(u+tv)^{n-1-2k}} w^k x^{n-1}=
\dfrac{1}
{1-[1]_{p,q}x-\dfrac{[1]_{p,q}[2]_{p,q}wx^2}
{1-[2]_{p,q}x-\dfrac{[2]_{p,q}[3]_{p,q}wx^2}
{1-[3]_{p,q}x-\dfrac{[3]_{p,q}[4]_{p,q}wx^2}
{\cdots}}}}.
\end{equation}
Since the right-hand side of the above identity is free of variables $t$, $u$, and $v$,
 the coefficient of $w^k x^{n-1}$ in the left-hand side
is a polynomial in $p$ and $q$ with nonnegative integral coefficients.  If we denote this coefficient by
$$
P_{n,k}(p,q):=\frac{a_{n,k}(p,q,t,u,v)}{t^k(u+tv)^{n-1-2k}},
$$
then
$P_{n,k}(p,q)=a_{n,k}(p,q,1,1,0)=a_{n,k}(p,q, 0, 1).$
On the other hand, comparing \eqref{eq:interpret1}
and  \eqref{eq:combin}, we see that
$a_{n,k}(p,q)=a_{n,k}(p,q, 0, 1)$.  Thus  $P_{n,k}(p,q)=a_{n,k}(p,q)$. This proves \eqref{eq:dfA}.
Finally, as
$(p+q) ~|~ [n]_{p,q} [n+1]_{p,q}$ for all $n\ge 1$,
each  $w$ appears with  a factor $(p+q)$ in the right-hand side of \eqref{eq:b},
and  the polynomial $P_{n,k}(p,q)$   is divisible by  $(p+q)^k$. \qed

\section{Proof of Theorem \ref{thm3}}\label{sec:bijection}
In our previous paper \cite[Theorem 8]{SZ10}, a  bijection $\Phi: \S_n\to  \S_n$  was constructed such  that
$$
(\RESS, \LES, \des, \fmax)\sigma=
(\nest, \cros, \defi, \fix) \Phi(\sigma)\qquad \forall \sigma\in \S_{n}.
$$
In this section we first recall  the bijection $\Phi$  and then show that the same bijection satisfies
\begin{align}\label{eq:key}
(\da-\fmax, \dd, \valley)\sigma=
( \cda, \cdd, \cvalley)\Phi(\sigma) \qquad \forall \sigma\in \S_{n}.
\end{align}

Let  $\sigma=\sigma(1)\ldots \sigma(n)$ be a permutation of $[n]$,
an {\em inversion top number} (resp. {\em inversion bottom number}) of a letter $i$
in the word $\sigma$ is the number of occurrences of inversions of form $(i,j)$ (resp $(j,i)$) in $\sigma$.
We now construct $\tau = \Phi(\sigma)$ in such a way that
$$
\231_k\sigma = \nest_k\tau \quad \forall k=1,\ldots,n.
$$
Given a permutation $\sigma$, we first construct two biwords, $f \choose f'$ and $g \choose g'$, and then form the biword $\tau=\left({f \atop f'}~{g \atop g'}\right)$ by concatenating $f$ and $g$, and $f'$ and $g'$, respectively.
The word $f$ is defined as the subword of descent bottoms in $\sigma$, ordered increasingly, and $g$ is defined as the subword of nondescent bottoms in $\sigma$, also ordered increasingly.
The word $f'$ is the permutation on descent tops in $\sigma$ such that the inversion bottom number of each letter $a$ in $f'$ is the right embracing number of $a$ in $\sigma$. Similarly, the word $g'$ is the permutation on nondescent tops in $\sigma$ such that the inversion top number of each letter $b$ in $g'$ is the right embracing number of $b$ in $\sigma$. Rearranging the columns of $\tau'$, so that the bottom row is in increasing order, we obtain the permutation $\tau=\Phi(\sigma)$ as the top row of the rearranged bi-word.

\begin{ex}
Let $\sigma=4~1~ 2~ 7~ 9~6~5~8~3$, with right embracing numbers $1,0,0,2,0,1,1,0,0$.
Then
$$
{f\choose f'}
= \left(
{1 \atop 8}~
{3 \atop 4}~
{5 \atop 6}~
{6 \atop 9}
\right),
\quad
{g\choose g'}
= \left(
{2 \atop 1}~
{4 \atop 2}~
{7 \atop 7}~
{8 \atop 5}~
{9 \atop 3}
\right),
$$

$$
\tau'
= \left( {f \atop f'}~{g \atop g'} \right)
= \left(
{1 \atop 8}~
{3 \atop 4}~
{5 \atop 6}~
{6 \atop 9}~
{2 \atop 1}~
{4 \atop 2}~
{7 \atop 7}~
{8 \atop 5}~
{9 \atop 3}
\right)
\to
\left(
{2 \atop 1}~
{4 \atop 2}~
{9 \atop 3}~
{3 \atop 4}~
{8 \atop 5}~
{5 \atop 6}~
{7 \atop 7}~
{1 \atop 8}~
{6 \atop 9}
\right).
$$
and thus $\Phi(\sigma)=\tau = 2~4~9~3~8~5~7~1~6$.
\end{ex}

%
%

Let  $\tau=\Phi(\sigma) \in \S_n$. We enumerate the triple statistics $(\da\sigma - \fmax \sigma, \dd\sigma, \valley\sigma)$.
\begin{itemize}
\item If $k$ is a double ascent of $\sigma$ and not a foremaximum of $\sigma$, then $k$ is a ascent top and also bottom. So $k$ belongs to $g$ and $g'$. By definition, $\toht_k\sigma>0$. Hence, the column ${k \choose k}$ does not appear in ${g \choose g'}$, i.e., $\tau^{-1}(k) < k < \tau(k)$. So $k$ is a cyclic double ascent of $\tau$. Conversely, if $\tau^{-1}(k) < k < \tau(k)$, then the column $k$ appears in $g$ and $g'$ and $\toht_k\sigma>0$. It implies that $k$ is a double ascent of $\sigma$ and not a foremaximum of $\sigma$.
\item If $k$ is a double descent of $\sigma$, then $k$ is a descent top and also bottom. So $k$ belongs to $f$ and $f'$, then $\tau^{-1}(k) > k > \tau(k)$. Hence $k$ is a cyclic double descent of $\tau$. Conversely, if $\tau^{-1}(k)>k>\tau(k)$, then the column $k$ appears in $f$ and $f'$. It implies that $k$ is a double descent of $\sigma$.
\item If $k$ is a valley of $\sigma$, then $k$ is a descent bottom and also ascent top. So $k$ belongs to $f$ and $g'$, then $\tau^{-1}(k) > k < \tau(k)$. Hence $k$ is a cyclic valley of $\tau$. Conversely, if $\tau^{-1}(k)>k<\tau(k)$, then the column $k$ appears in $f$ and $g'$. It implies that $k$ is a valley of $\sigma$
\end{itemize}
Thus \eqref{eq:key} is established. The proof of  Theorem~\ref{thm3} is then completed. \qed

We illustrate $\Phi: \S_4\to \S_4$ with their statistics in Figure~\ref{fig:table1}. 

\begin{figure}[t]
$$
{\small
\begin{tabular}{c|c||c|c|c|c|c|c|c}
\multirow{2}{*}{$\sigma \in \S_4$} &
\multirow{2}{*}{$\tau=\Phi(\sigma)$}
& $ \des\sigma$ & $\les\sigma$ & $\ress\sigma$ & $ \da\sigma - \fmax \sigma$ & $\dd\sigma$ & $ \valley\sigma$ & $\fmax \sigma$\\ \cline{3-9}
& & $\defi\tau$ & $\cros\tau$ & $\nest\tau$ & $\cda\tau$ & $\cdd\tau$ & $\cvalley\tau$ & $\fix \tau$ \\ \hline
1234 & 1234 & 0 & 0 & 0 & 0 & 0 & 0 & 4 \\
1243 & 1243 & 1 & 0 & 0 & 0 & 0 & 1 & 2 \\
1324 & 1324 & 1 & 0 & 0 & 0 & 0 & 1 & 2 \\
1342 & 1432 & 1 & 0 & 1 & 0 & 0 & 1 & 2 \\
1423 & 1342 & 1 & 1 & 0 & 1 & 0 & 1 & 1 \\
1432 & 1423 & 2 & 0 & 0 & 0 & 1 & 1 & 1 \\
2134 & 2134 & 1 & 0 & 0 & 0 & 0 & 1 & 2 \\
2143 & 2143 & 2 & 0 & 0 & 0 & 0 & 2 & 0 \\
2314 & 3214 & 1 & 0 & 1 & 0 & 0 & 1 & 2 \\
2341 & 4231 & 1 & 0 & 2 & 0 & 0 & 1 & 2 \\
2413 & 3241 & 1 & 1 & 1 & 1 & 0 & 1 & 1 \\
2431 & 4213 & 2 & 0 & 1 & 0 & 1 & 1 & 1 \\
3124 & 2314 & 1 & 1 & 0 & 1 & 0 & 1 & 1 \\
3142 & 3421 & 2 & 1 & 1 & 0 & 0 & 2 & 0 \\
3214 & 3124 & 2 & 0 & 0 & 0 & 1 & 1 & 1 \\
3241 & 4321 & 2 & 0 & 2 & 0 & 0 & 2 & 0 \\
3412 & 2431 & 1 & 1 & 1 & 1 & 0 & 1 & 1 \\
3421 & 4132 & 2 & 0 & 1 & 0 & 1 & 1 & 1 \\
4123 & 2341 & 1 & 2 & 0 & 2 & 0 & 1 & 0 \\
4132 & 3412 & 2 & 2 & 0 & 0 & 0 & 2 & 0 \\
4213 & 3142 & 2 & 1 & 0 & 1 & 1 & 1 & 0 \\
4231 & 4312 & 2 & 1 & 1 & 0 & 0 & 2 & 0 \\
4312 & 2413 & 2 & 1 & 0 & 1 & 1 & 1 & 0 \\
4321 & 4123 & 3 & 0 & 0 & 0 & 2 & 1 & 0 \\
\end{tabular}
}
$$
\caption{Illustration of $\Phi$ on $\S_4$ with their statistics}
\label{fig:table1}
\end{figure}

\section{Proof of Theorem~\ref{thm:b}}
Let $\sigma\in \S_{n}$, the refinements of crossing and nesting are defined by
\begin{align*}
\cros_k\sigma &= \# \set{i \in [n]: (i<k\le\sigma(i)<\sigma(k))\vee(i>k>\sigma(i)>\sigma(k))},\\
\nest_k\sigma &= \# \set{i \in [n]: (i<k\le\sigma(k)<\sigma(i))\vee(i>k>\sigma(k)>\sigma(i))}.
\end{align*}
We clearly have $\cros\sigma=\sum_{k=1}^{n}\cros_{k}\sigma$ and $\nest\sigma=\sum_{k=1}^{n}\nest_{k}\sigma$.

Using Foata-Zeilberger's bijection$\Psi_{FZ} : \S_{n} \to \H_{n}$, we construct the Laguerre history
$(s_{0}, \ldots, s_{n}, p_{1},\ldots, p_{n})$, where $s_{0}=(0,0)$ and the step
$(s_{i-1}, s_{i})$ is North-East, South-East, East blue and East red
if $i$ is a cyclic valley, cyclic peak, cyclic double ascent (or fixed point), or cyclic double descent, respectively; while $p_{i}= \nest_{i}\sigma$ for $i=1,\ldots, n$.
Then, we have
\[
\nest_{i}\sigma+\cros_{i}\sigma=
\left\{
  \begin{array}{ll}
    h_{i}, & \text{if $(s_{i-1}, s_{i})$ is North-East;} \\
    h_{i}-1, & \text{if $(s_{i-1}, s_{i})$ is South-East;} \\
    h_{i}, & \text{if $(s_{i-1}, s_{i})$ is East blue;} \\
    h_{i}-1, & \text{if $(s_{i-1}, s_{i})$ is East red.}
  \end{array}
\right.
\]
Thus $(s_{0}, \ldots, s_{n}, p_{1},\ldots, p_{n})$ is a Laguerre history of length $n$ and
$$
w(\sigma)=t^{\ER \gamma+\NE \gamma}u^{\EB\gamma}v^{\ER\gamma}w^{\NE \gamma} y^{\EB^* \gamma} q^{\NE \gamma + \EB \gamma}\prod_{i=1}^{n}p^{p_{i}}q^{h_{i}-1-p_{i}},
$$
where $\NE \gamma$, $\EB \gamma$, and $\ER\gamma$ are the number of North-East steps, East blue steps, and East red steps of $\gamma$ and $\EB^* \gamma$ is the number of East blue steps whose height is equal to $p_i$.
Given a Motzkin path $\gamma$, weight each step at height $h$ by
\begin{align}\label{eq:weight_D}
a_k:=tw[h+1]_{p,q},\quad b_k := yp^h+(qu+tv)[h]_{p,q},\quad c_k := [h]_{p,q}
\end{align}
if the step is North-East, East, and South-East, respectively,
and the weight of $\gamma$ is defined to be the product of  the step weights. Then the last sum amounts to sum over all the Motzkin paths
of length $n-1$ with respect to \eqref{eq:weight_D}, i.e.,
\begin{align}\label{eq:gf2_D}
B_n(p,q,t,u,v,w,y) &= \sum_{\gamma \in \M_{n}} w(\gamma).
\end{align}
By \eqref{eq:fla}, we have the following continued fraction expansion:
\begin{equation}\label{eq:A_D}
 \sum_{n\ge 0} B_n(p,q,t,u,v,w,y)x^n
=
\dfrac{1}{
1 - b_0 x - \dfrac{a_0 c_1 x^2}{
1 - b_1 x - \dfrac{a_1 c_2 x^2}{
\cdots}}},
\end{equation}
where $a_h = tw[h+1]_{p,q}$, $b_h=yp^h+(qu+tv)[h]_{p,q}$, and $c_h=[h]_{p,q}$.

For $0\leq k\leq {\floor{(n-d)/2}}$,
let $b_{n,k,d}(p,q,t,u,v)$ be the coefficient of $w^ky^{d}$ in
$B_n(p,q,t,u,v,w,y)$,
that is,
\begin{equation}
B_n(p,q,t,u,v,w,y)=\sum_{d=0}^{n} \sum_{k=0}^{\floor{(n-d)/2}} b_{n,k,d}(p,q,t,u,v) w^ky^{d}.
\label{eq:interpret1_D}
\end{equation}
Substituting $x \leftarrow \frac{x}{(qu+tv)}$, $w \leftarrow \frac{w(qu+tv)^2}{t}$, and $y\leftarrow (qu+tv)y$ in  \eqref{eq:A_D},
we obtain
\begin{align*}
&\sum_{n\ge0}\sum_{d=0}^{n}\sum_{k=0}^{\floor{(n-d)/2}} \frac{b_{n,k,d}(p,q,u,v,t)}{t^k(qu+tv)^{n-2k-d}}w^k y^{d}x^{n}\\
&=
\dfrac{1}
{1-(y+[0]_{p,q})x-\dfrac{[1]_{p,q}^2wx^2}
{1-(py+[1]_{p,q})x-\dfrac{[2]_{p,q}^2wx^2}
{1-(p^2y+[2]_{p,q})x-\dfrac{[3]_{p,q}^2wx^2}
{\cdots}}}}.
\end{align*}
Since the right-hand side of the above identity  is free of variables $t$, $u$, and $v$, the coefficient of $w^ky^{d}x^{n}$
in the left-hand side  is a polynomial in $p$ and $q$ with
 nonnegative integral coefficients.
Denote this polynomial by $Q_{n,k,d}(p,q)$, then
\begin{equation}
Q_{n,k,d}(p,q):=\frac{b_{n,k,d}(p,q,t, u,v)}{t^k(qu+tv)^{n-2k-d}}.
\label{eq:interpret2_D}
\end{equation}
Hence
$Q_{n,k,d}(p,q)=b_{n,k,d}(p,q,1,0,1).$
By \eqref{eq:defB} and \eqref{eq:interpret1_D}, we derive
that
$Q_{n,k,d}(p,q)=b_{n,k,d}(p,q)$,
which is  given by \eqref{eq:combin_D}. This completes the proof. \qed

\medskip
\begin{rmk}
Let $A_n(p,q,t)=A_n(p,q,t,1,1,1)$ and $B_n(p,q,t)=B_n(p,q,t,1,1,1,1)$.
From \eqref{eq:A} and \eqref{eq:A_D} we derive
\begin{align}\label{eq:cfrac1}
\sum_{n\ge0} A_n(p,q,t)x^{n}
=
1+\dfrac{[1]_{p,q}x}
{1-(1+t)[1]_{p,q}x-\dfrac{[1]_{p,q}[2]_{p,q}tx^2}
{1-(1+t)[2]_{p,q}x-\dfrac{[2]_{p,q}[3]_{p,q}tx^2}
{\cdots}}},
\end{align}
and
\begin{align}\label{eq:cfrac2}
\sum_{n\ge 0} B_n(p,q,t)x^n
= \dfrac{1}{
1 - (t[0]_{p,q}+[1]_{p,q}) x - \dfrac{[1]_{p,q}^2 tx^2}{
1 - (t[1]_{p,q}+[2]_{p,q}) x - \dfrac{[2]_{p,q}^2 tx^2}{
\cdots}}}.
\end{align}
In view of the  following
contraction formulae with $c_{2i-1}=[i]_{p,q}$ and $c_{2i}=t[i]_{p,q}$ for $i\geq 1$:
\begin{align*}
\cfrac{1}{1-\cfrac{c_1x}{1-\cfrac{c_2x}{\cdots}}}
&=1+\cfrac{c_1x}{
1-(c_1+c_2)x-\cfrac{c_2c_3x^2}{
1-(c_3+c_4)x-\cfrac{c_4c_5x^2}{
\cdots}}}\\
&=\cfrac{1}{
1-c_1x-\cfrac{c_1c_2x^2}{
1-(c_2+c_3)x-\cfrac{c_3c_4x^2}{
\cdots}}},
\end{align*}
we see that
\begin{align}
\sum_{n\geq 0}A_n(p,q,t)x^n=\sum_{n\geq 0}B_n(p,q,t)x^n
=\cfrac{1}{1-\cfrac{c_1 x}{1-\cfrac{c_2 x}{1-\cfrac{c_3 x}{\cdots}}}},
\end{align}
where $c_{2i}=t[i]_q$ and $c_{2i-1}=[i]_q$.
Hence $A_n(p,q,t)=B_n(p,q,t)$,  i.e., the two triple statistics $(\RES, \LES, \des)$ and $(\RESS, \LES, \des)$ are equidistributed on $\S_n$.
Although a bijection  showing the latter  equidistribution can be constructed  by combining the known bijections from permutations to
Motzkin paths,  a direct  bijective proof of the above equidistribution  is desired.
\end{rmk}

\section{Proof of Theorem~\ref{thm: dcycle}}\label{sec: dcycle}
Consider the polynomial
\begin{align}\label{eq:defD}
C_{n}(\beta,t,u,v,w):=\sum_{\sigma\in \D_{n}} \beta^{\cyc \sigma} t^{\exc \sigma} u^{\cda \sigma} v^{\cdd \sigma} w^{\cvalley \sigma}
=\sum_{k=0}^{\floor{n/2}}c_{n,k}(\beta,t,u,v)w^k.
\end{align}
From \cite{Zen93} we know that
\begin{multline} \label{eq:D}
1+\sum_{n\ge 1}C_{n}(\beta,t,u,v,w)x^n= \\
\dfrac{1}{
1-0(tu+v)x-\dfrac{1(\beta+0)t w x^2}{
1-1(tu+v)x-\dfrac{2(\beta+1)t w x^2}{
1-2(tu+v)x-\dfrac{3(\beta+2)t w x^2}{\cdots
}}}}.
\end{multline}
Substituting $x \leftarrow \frac{x}{(tu+v)}$ and $w \leftarrow \frac{w(tu+v)^2}{t}$ in  \eqref{eq:D}, we get
\begin{align}\label{eq:ccf}
1+\sum_{n\ge 1} \sum_{k=0}^{\floor{n/2}} \frac{c_{n,k}(\beta,t,u,v)}{t^k(tu+v)^{n-2k}} w^k x^n=
\dfrac{1}{
1-\dfrac{1(\beta+0) w x^2}{
1-x-\dfrac{2(\beta+1) w x^2}{
1-2x-\dfrac{3(\beta+2) w x^2}{
1-3x-\dfrac{4(\beta+3) w x^2}{\cdots
}}}}}.
\end{align}
By the right-hand side we see  that  $c_{n,k}^*(\beta):=\frac{c_{n,k}(\beta,t,u,v)}{t^k(tu+v)^{n-2k}}$ must be a polynomial in $\beta$, which
is  independent of $u$, $v$, and $t$.  Hence we have
\begin{align}\label{eq:last}
C_{n}(\beta,t,u,v,w)=\sum_{k=0}^{\floor{n/2}} c_{n,k}^*(\beta) t^k(tu+v)^{n-2k} w^k.
\end{align}
Thus
$C_n(\beta,1,1,0,w)=\sum_{k=0}^{\floor{n/2}} c_{n,k}^*(\beta) w^k$.
On the other side, by  \eqref{def:d}  and \eqref{eq:defD} we have
 $C_{n}(\beta,1,1,0,w)=\sum_{k=0}^{\floor{n/2}}c_{n,k}(\beta)w^k$. So $c_{n,k}(\beta)=c_{n,k}^*(\beta)$.
Finally we get \eqref{eq:d} by putting $u=v=w=1$ in \eqref{eq:last}.
\qed
\section{Two star  variations}\label{sec:new}

We show that the polynomial $A_n$, which is originally defined using \emph{linear statistics}, has a new combinatorial interpretation by \emph{cyclic statistics}.
For $\sigma=\sigma(1)\dots\sigma(n) \in \S_n$ and some integral function `$\stat$' on $\S_n$, we define
the \emph{star} transformation from the permutation $\sigma$ to the function $\sigma^*=\sigma^*(1)\dots\sigma^*(n)$ from $[n]$ to $\set{0,\dots,n-1}$ by
\begin{equation}
\label{eq:star}
\sigma \mapsto \sigma^*=(\sigma(1)-1)\dots(\sigma(n)-1).
\end{equation}
For any statistic $\stat$ on $\sigma^*$ we can define the corresponding star statistic
 `$\stat^{*}$' on $\sigma$ by  $\stat^{*}(\sigma) := \stat(\sigma^{*})$.
For instance, we can define the \emph{star} cyclic statistics
\begin{align*}
\fix^{*}\sigma &= \fix\sigma^{*} = \# \set{i\in [n-1]: i=\sigma^{*}(i)=\sigma(i)-1},\\
\wex^{*}\sigma &= \wex\sigma^{*} = \# \set{i\in [n-1]: i\le\sigma^{*}(i)=\sigma(i)-1} (= \exc \sigma),\\
\cros^{*}\sigma &= \cros\sigma^{*} = \# \set{(i,j)\in [n]\times[n]: (i<j\le\sigma^{*}(i)<\sigma^{*}(j))\vee(i>j>\sigma^{*}(i)>\sigma^{*}(j))},\\
\nest^{*}\sigma &= \nest\sigma^{*} = \# \set{(i,j)\in [n]\times[n]: (i<j\le\sigma^{*}(j)<\sigma^{*}(i))\vee(i>j>\sigma^{*}(j)>\sigma^{*}(i))},\\
\cdd^{*}\sigma &= \cdd\sigma^{*} = \# \set{i \in [n-1]: (\sigma^{*})^{-1}(i)>i>\sigma^{*}(i)},\\
\cda^{*}\sigma &= \cda\sigma^{*} = \# \set{i \in [n-1]: (\sigma^{*})^{-1}(i)<i<\sigma^{*}(i)},\\
\cvalley^{*}\sigma &= \cvalley\sigma^{*} = \# \set{i \in [n-1]:  (\sigma^{*})^{-1}(i)>i<\sigma^{*}(i)}.
\end{align*}
Clearly, $\sigma^{*}(0)$ and $(\sigma^{*})^{-1}(n)$ are neither defined nor needed.

\begin{rmk}

The star crossing number $\cros^*(\sigma)$ of $\sigma$ is different with the ordinary crossing number $\cros(\sigma)$ of $\sigma$.
For example, for the  previous example $\sigma=3762154$, we have $\sigma^*=2651043=\left({1 \atop 2}{2 \atop 6}{3 \atop 5}{4 \atop 1}{5 \atop 0}{6 \atop 4}{7 \atop 3}\right)$.
$$
\centering
\begin{pgfpicture}{78.00mm}{60.14mm}{172.00mm}{92.50mm}
\pgfsetxvec{\pgfpoint{1.00mm}{0mm}}
\pgfsetyvec{\pgfpoint{0mm}{1.00mm}}
\color[rgb]{0,0,0}\pgfsetlinewidth{0.30mm}\pgfsetdash{}{0mm}
\pgfmoveto{\pgfxy(80.00,80.00)}\pgflineto{\pgfxy(170.00,80.00)}\pgfstroke
\pgfsetlinewidth{0.15mm}\pgfmoveto{\pgfxy(120.00,80.00)}\pgfcurveto{\pgfxy(121.04,82.19)}{\pgfxy(122.81,83.96)}{\pgfxy(125.00,85.00)}\pgfcurveto{\pgfxy(130.53,87.63)}{\pgfxy(137.15,85.42)}{\pgfxy(140.00,80.00)}\pgfstroke
\pgfmoveto{\pgfxy(140.00,80.00)}\pgfcurveto{\pgfxy(138.54,78.14)}{\pgfxy(136.86,76.46)}{\pgfxy(135.00,75.00)}\pgfcurveto{\pgfxy(121.14,64.14)}{\pgfxy(101.14,66.37)}{\pgfxy(90.00,80.00)}\pgfstroke
\pgfmoveto{\pgfxy(130.00,80.00)}\pgfcurveto{\pgfxy(128.71,78.00)}{\pgfxy(127.00,76.29)}{\pgfxy(125.00,75.00)}\pgfcurveto{\pgfxy(116.68,69.65)}{\pgfxy(105.62,71.86)}{\pgfxy(100.00,80.00)}\pgfstroke
\pgfmoveto{\pgfxy(150.00,80.00)}\pgfcurveto{\pgfxy(148.96,77.81)}{\pgfxy(147.19,76.04)}{\pgfxy(145.00,75.00)}\pgfcurveto{\pgfxy(139.47,72.37)}{\pgfxy(132.85,74.58)}{\pgfxy(130.00,80.00)}\pgfstroke
\pgfputat{\pgfxy(120.00,76.00)}{\pgfbox[bottom,left]{\fontsize{11.38}{13.66}\selectfont \makebox[0pt]{$3$}}}
\pgfputat{\pgfxy(100.00,76.00)}{\pgfbox[bottom,left]{\fontsize{11.38}{13.66}\selectfont \makebox[0pt]{$1$}}}
\pgfputat{\pgfxy(110.00,76.00)}{\pgfbox[bottom,left]{\fontsize{11.38}{13.66}\selectfont \makebox[0pt]{$2$}}}
\pgfputat{\pgfxy(130.00,76.00)}{\pgfbox[bottom,left]{\fontsize{11.38}{13.66}\selectfont \makebox[0pt]{$4$}}}
\pgfputat{\pgfxy(140.00,76.00)}{\pgfbox[bottom,left]{\fontsize{11.38}{13.66}\selectfont \makebox[0pt]{$5$}}}
\pgfputat{\pgfxy(90.00,76.00)}{\pgfbox[bottom,left]{\fontsize{11.38}{13.66}\selectfont \makebox[0pt]{$0$}}}
\pgfcircle[fill]{\pgfxy(130.36,72.01)}{1.00mm}
\pgfsetlinewidth{0.30mm}\pgfcircle[stroke]{\pgfxy(130.36,72.01)}{1.00mm}
\pgfcircle[fill]{\pgfxy(110.00,80.00)}{1.00mm}
\pgfcircle[stroke]{\pgfxy(110.00,80.00)}{1.00mm}
\pgfputat{\pgfxy(125.00,63.00)}{\pgfbox[bottom,left]{\fontsize{11.38}{13.66}\selectfont \makebox[0pt]{$\sigma^*=2651043$}}}
\pgfmoveto{\pgfxy(129.50,86.50)}\pgflineto{\pgfxy(131.00,86.00)}\pgflineto{\pgfxy(129.50,85.50)}\pgfclosepath\pgffill
\pgfmoveto{\pgfxy(129.50,86.50)}\pgflineto{\pgfxy(131.00,86.00)}\pgflineto{\pgfxy(129.50,85.50)}\pgfclosepath\pgfstroke
\pgfmoveto{\pgfxy(140.50,70.50)}\pgflineto{\pgfxy(139.00,70.00)}\pgflineto{\pgfxy(140.50,69.50)}\pgfclosepath\pgffill
\pgfmoveto{\pgfxy(140.50,70.50)}\pgflineto{\pgfxy(139.00,70.00)}\pgflineto{\pgfxy(140.50,69.50)}\pgfclosepath\pgfstroke
\pgfmoveto{\pgfxy(104.50,85.50)}\pgflineto{\pgfxy(106.00,85.00)}\pgflineto{\pgfxy(104.50,84.50)}\pgfclosepath\pgffill
\pgfmoveto{\pgfxy(104.50,85.50)}\pgflineto{\pgfxy(106.00,85.00)}\pgflineto{\pgfxy(104.50,84.50)}\pgfclosepath\pgfstroke
\pgfmoveto{\pgfxy(115.50,68.50)}\pgflineto{\pgfxy(114.00,68.00)}\pgflineto{\pgfxy(115.50,67.50)}\pgfclosepath\pgffill
\pgfmoveto{\pgfxy(115.50,68.50)}\pgflineto{\pgfxy(114.00,68.00)}\pgflineto{\pgfxy(115.50,67.50)}\pgfclosepath\pgfstroke
\pgfmoveto{\pgfxy(115.50,72.50)}\pgflineto{\pgfxy(114.00,72.00)}\pgflineto{\pgfxy(115.50,71.50)}\pgfclosepath\pgffill
\pgfmoveto{\pgfxy(115.50,72.50)}\pgflineto{\pgfxy(114.00,72.00)}\pgflineto{\pgfxy(115.50,71.50)}\pgfclosepath\pgfstroke
\pgfputat{\pgfxy(150.00,76.00)}{\pgfbox[bottom,left]{\fontsize{11.38}{13.66}\selectfont \makebox[0pt]{$6$}}}
\pgfsetlinewidth{0.15mm}\pgfmoveto{\pgfxy(100.00,80.00)}\pgfcurveto{\pgfxy(100.00,82.76)}{\pgfxy(102.24,85.00)}{\pgfxy(105.00,85.00)}\pgfcurveto{\pgfxy(107.76,85.00)}{\pgfxy(110.00,82.76)}{\pgfxy(110.00,80.00)}\pgfstroke
\pgfmoveto{\pgfxy(110.00,80.00)}\pgfcurveto{\pgfxy(111.40,81.91)}{\pgfxy(113.09,83.60)}{\pgfxy(115.00,85.00)}\pgfcurveto{\pgfxy(126.09,93.10)}{\pgfxy(141.62,90.88)}{\pgfxy(150.00,80.00)}\pgfstroke
\pgfmoveto{\pgfxy(129.50,90.50)}\pgflineto{\pgfxy(131.00,90.00)}\pgflineto{\pgfxy(129.50,89.50)}\pgfclosepath\pgffill
\pgfsetlinewidth{0.30mm}\pgfmoveto{\pgfxy(129.50,90.50)}\pgflineto{\pgfxy(131.00,90.00)}\pgflineto{\pgfxy(129.50,89.50)}\pgfclosepath\pgfstroke
\pgfcircle[fill]{\pgfxy(125.00,75.00)}{1.00mm}
\pgfcircle[stroke]{\pgfxy(125.00,75.00)}{1.00mm}
\pgfputat{\pgfxy(160.00,76.00)}{\pgfbox[bottom,left]{\fontsize{11.38}{13.66}\selectfont \makebox[0pt]{$7$}}}
\pgfsetlinewidth{0.15mm}\pgfmoveto{\pgfxy(160.00,80.00)}\pgfcurveto{\pgfxy(158.60,78.09)}{\pgfxy(156.91,76.40)}{\pgfxy(155.00,75.00)}\pgfcurveto{\pgfxy(143.91,66.90)}{\pgfxy(128.38,69.12)}{\pgfxy(120.00,80.00)}\pgfstroke
\pgfmoveto{\pgfxy(140.50,74.50)}\pgflineto{\pgfxy(139.00,74.00)}\pgflineto{\pgfxy(140.50,73.50)}\pgfclosepath\pgffill
\pgfsetlinewidth{0.30mm}\pgfmoveto{\pgfxy(140.50,74.50)}\pgflineto{\pgfxy(139.00,74.00)}\pgflineto{\pgfxy(140.50,73.50)}\pgfclosepath\pgfstroke
\pgfcircle[fill]{\pgfxy(135.25,75.07)}{1.00mm}
\pgfcircle[stroke]{\pgfxy(135.25,75.07)}{1.00mm}
\end{pgfpicture}%
$$
Since the four pairs $(1<2\le\sigma^*(1)<\sigma^*(2))$, $(6>5>\sigma^*(6)>\sigma^*(5))$, $(7>4>\sigma^*(7)>\sigma^*(4))$, and $(7>5>\sigma^*(7)>\sigma^*(5))$ are only crossings in $\sigma^*$,
$\cros^*(\sigma)=\cros(\sigma^*)=4$.
On the other hand, the three pairs $(2<3\le\sigma^*(3)<\sigma^*(2))$, $(5>4>\sigma^*(4)>\sigma^*(5))$, and $(7>6>\sigma^*(6)>\sigma^*(7))$ are only nestings in $\sigma^*$,
thus $\nest^*(\sigma)=\nest(\sigma^*)=3$.
\end{rmk}

\begin{thm}\label{thm:new}
The
two hextuple statistics
$(\RES, \LES, \des, \da^*, \dd^*, \valley^*)$ and
$(\nest^{*}, \linebreak[2] \cros^{*}, \linebreak[2] \defi^{*}-1, \linebreak[2] \linebreak[2] \cda^{*}+\fix^{*}, \cdd^{*}, \cvalley^{*})$  are equidistributed on $\S_n$.
In other words, we have
\begin{align}
A_n(p,q,t,u,v,w) = \sum_{\sigma\in \S_n} p^{\nest^{*} \sigma} q^{\cros^{*} \sigma} t^{\defi^{*}\sigma-1} u^{\cda^{*}\sigma+\fix^{*}\sigma} v^{\cdd^{*}\sigma} w^{\cvalley^{*}\sigma}.\label{eq:q-Euler_Cy}
\end{align}
\end{thm}

%
\begin{proof}
Using the bijection $\Phi$ in Section \ref{sec:bijection} (cf. Theorem 5), we can give a new bijection $\Psi$ on $\S_n$.
Given a permutation $\sigma=\sigma(1)\dots\sigma(n) \in \S_n$, let $\hat{\sigma}=(\sigma(1)+1)\dots(\sigma(n)+1)(1)\in \S_{n+1}$ and
consider the permutation $\tau=\Phi(\hat{\sigma}) \in \S_{n+1}$.
Since the last element $\hat{\sigma}(n+1)$ of $\hat{\sigma}$ is $1$, the first element $\tau(1)$ of $\tau$ should be $n+1$. So $\tau(2)\dots\tau(n+1)$ is a permutation on $[n]$ and we can define the permutation
$\Psi(\sigma):=\tau(2)\dots\tau(n+1) \in \S_{n}.$
By  Theorem{~}\ref{thm3},  the bijection $\Phi$ satisfies $$(\RESS, \LES, \des)\hat{\sigma}=(\nest, \cros, \defi)\Phi(\hat{\sigma}).$$
Since $(\RES, \LES, \des)\sigma = (\RESS+\des-n, \LES, \des-1)\hat{\sigma}$ and
$$(\nest^*, \cros^*, \defi^*-1)\Psi(\sigma)=(\nest+\defi-n, \cros, \defi-1)\Phi(\hat{\sigma}),$$ we have
$(\RES, \LES, \des)\sigma = (\nest^*, \cros^*, \defi^*-1)\Psi(\sigma).$
Hence, assuming $\sigma(0)=\sigma(n+1)=0$,  the two hextuple statistics $(\RES, \LES, \des, \da^*, \dd^*, \valley^*)$ and $(\nest^*, \cros^*, \linebreak[2] \defi^*-1, \cdd^*, $
$\cda^*+\fix^*, \cvalley^*)$ are equidistributed on $\S_n$.
\end{proof}

\begin{ex}
Given $\sigma=4~1~ 2~ 7~ 9~6~5~8~3$, we define $\hat{\sigma}=5~2~3~8~10~7~6~9~4~1$ with right embracing numbers $1,1,1,2,0,1,1,0,0,0$.
Then
$$
{f\choose f'}
= \left(
{1 \atop 4}~
{2 \atop 9}~
{4 \atop 5}~
{6 \atop 7}~
{7 \atop 10}
\right),
\quad
{g\choose g'}
= \left(
{3 \atop 2}~
{5 \atop 3}~
{8 \atop 8}~
{9 \atop 6}~
{10 \atop 1}
\right),
$$

$$
\tau'
= \left( {f \atop f'}~{g \atop g'} \right)
= \left(
{1 \atop 4}~
{2 \atop 9}~
{4 \atop 5}~
{6 \atop 7}~
{7 \atop 10}~
{3 \atop 2}~
{5 \atop 3}~
{8 \atop 8}~
{9 \atop 6}~
{10 \atop 1}
\right)
\to
\left(
{10 \atop 1}~
{3 \atop 2}~
{5 \atop 3}~
{1 \atop 4}~
{4 \atop 5}~
{9 \atop 6}~
{6 \atop 7}~
{8 \atop 8}~
{2 \atop 9}~
{7 \atop 10}
\right),
$$
and thus $\tau=\Phi(\hat{\sigma})= 10~3~5~1~4~9~6~8~2~7$. So, $\Psi(\sigma)=\tau(2)\dots\tau(10)=3~5~1~4~9~6~8~2~7$.
\end{ex}

Let $\S_n^*$ be the set of all sequences $\sigma^*$ where $\sigma \in \S_n$, that is,
$\S_n^* := \set{\sigma^*:\sigma \in \S_n}.$
By definition \eqref{eq:star}, for all $\sigma \in \S_n$,
it holds that
\begin{align}
(\fix^*, \nest^*, \cros^*, \wex^*, \cdd^*, \cda^*, \cvalley^*) \sigma = (\fix, \nest, \cros, \wex, \cdd, \cda \cvalley)\sigma^*.
\end{align}
Note that, for a given $\sigma^* \in \S_n^*$, the entries $0$ and $n$ are neither a cyclic valley nor a cyclic peak.
We illustrate $\Psi: \S_4\to \S_4$ with their statistics in Figure~\ref{fig:table2}. 

\begin{figure}[t]
{\small
\begin{tabular}{c|c|c||c|c|c|c|c|c}
  \multirow{3}{*}{$\sigma \in \S_4$} &
  \multirow{3}{*}{$\tau=\Psi(\sigma)$ } &
  \multirow{3}{*}{$\tau^*\in\S_4^*$} &
  $\des\sigma$ & $\les\sigma$ & $\res\sigma$ & $\da^*\sigma$ & $\dd^*\sigma$ & $\valley^*\sigma$
  \\ \cline{4-9}
  & & &
  $\defi^{*} \tau -1$ & $\cros^{*}\tau$& $\nest^{*} \tau$ & $\cda^{*}\tau + \fix^{*}\tau$ & $\cdd^{*}\tau$ & $\cvalley^{*} \tau$
  \\ \cline{4-9}
  & & &
  $\defi \tau^{*}-1$ & $\cros \tau^{*}$& $\nest \tau^{*}$ & $\cda \tau^{*} + \fix \tau^{*}$ & $\cdd \tau^{*}$ & $\cvalley \tau^{*}$
  \\
  \hline
1234 & 2341 & 1230 & 0 & 0 & 0 & 3 & 0 & 0  \\
1243 & 2314 & 1203 & 1 & 0 & 0 & 2 & 1 & 0  \\
1324 & 2431 & 1320 & 1 & 0 & 1 & 1 & 0 & 1  \\
1342 & 2143 & 1032 & 1 & 0 & 0 & 2 & 1 & 0 \\
1423 & 2413 & 1302 & 1 & 1 & 0 & 1 & 0 & 1 \\
1432 & 2134 & 1023 & 2 & 0 & 0 & 1 & 2 & 0 \\
2134 & 3241 & 2130 & 1 & 0 & 1 & 1 & 0 & 1 \\
2143 & 3214 & 2103 & 2 & 0 & 1 & 0 & 1 & 1 \\
2314 & 4321 & 3210 & 1 & 0 & 2 & 1 & 0 & 1 \\
2341 & 1342 & 0231 & 1 & 0 & 0 & 2 & 1 & 0 \\
2413 & 4312 & 3201 & 1 & 1 & 1 & 1 & 0 & 1 \\
2431 & 1324 & 0213 & 2 & 0 & 0 & 1 & 2 & 0 \\
3124 & 3421 & 2310 & 1 & 1 & 1 & 1 & 0 & 1 \\
3142 & 3124 & 2013 & 2 & 1 & 1 & 0 & 1 & 1 \\
3214 & 4231 & 3120 & 2 & 0 & 2 & 0 & 1 & 1 \\
3241 & 1432 & 0321 & 2 & 0 & 1 & 0 & 1 & 1 \\
3412 & 3142 & 2031 & 1 & 1 & 0 & 1 & 0 & 1 \\
3421 & 1243 & 0132 & 2 & 0 & 0 & 1 & 2 & 1 \\
4123 & 3412 & 2301 & 1 & 2 & 0 & 1 & 0 & 1 \\
4132 & 4123 & 3012 & 2 & 2 & 0 & 0 & 1 & 1 \\
4213 & 4213 & 3102 & 2 & 1 & 1 & 0 & 1 & 1 \\
4231 & 1423 & 0312 & 2 & 1 & 0 & 0 & 1 & 1 \\
4312 & 3124 & 2013 & 2 & 1 & 0 & 0 & 1 & 1 \\
4321 & 1234 & 0123 & 3 & 0 & 0 & 0 & 3 & 0 \\
\end{tabular}
}
\caption{Illustration of $\Psi$ on $\S_4$ with their statistics}
\label{fig:table2}
\end{figure}

Given a $\sigma \in \S_n$, the diagram of $\sigma^*\in\S_n^*$ consists of cycles $i\rightarrow \sigma^*(i)\rightarrow \cdots \rightarrow i$ with $i\in [n]$, and the path $n\rightarrow \sigma^*(n)\rightarrow \cdots \rightarrow 0$.
Let  $\cyc^*\sigma$ be the number of cycles in the diagram of $\sigma^{*}$.
For example, for the previous example $\sigma=3762154$ and $\sigma^*=2651043$, we have $\cyc^*(\sigma)=\cyc(\sigma^*)=2$, since one cycle $1\to 2\to 6 \to 4$ and one path $7 \to 3 \to 5 \to 0$ exist in $\sigma^*$.
Let  $\S_{n}(k)$ be the set of permutations $\sigma\in \S_{n}$ with $\cvalley^* \sigma=k$ and $\cdd^* \sigma=0$.
Introduce the polynomial
\begin{align}\label{def:c}
d_{n,k}(\beta)=\sum_{\sigma \in \S_{n}(k)} \beta^{\cyc^* \sigma-\fix^*\sigma}.
\end{align}

\begin{thm} We have
\begin{align}\label{eq:c}
\sum_{\sigma\in \S_{n}} \beta^{\cyc^* \sigma-\fix^* \sigma} t^{\exc \sigma}=\sum_{k=0}^{\floor{(n-1)/2}} d_{n,k}(\beta)t^{k}(1+t)^{n-1-2k}.
\end{align}
Moreover,  for all $k\ge1$, the polynomial $d_{n,k}(\beta)$ has a factor $(\beta+1)$.
\end{thm}
\begin{proof}
Let
\begin{align}\label{eq:defD}
D_{n}(\beta,t,u,v,w):=\sum_{\sigma\in \S_{n}} \beta^{\cyc^* \sigma - \fix^*\sigma} t^{\wex^* \sigma} u^{\cda^* \sigma + \fix^* \sigma} v^{\cdd^* \sigma}
w^{\cvalley^* \sigma}.
\end{align}
By  the same method in \cite{Zen93} to count the cycles, we obtain
\begin{multline} \label{eq:C}
\sum_{n\ge1} D_n(\beta,t,u,v,w)x^{n-1}=\\
\dfrac{1}
{1-1(tu+v)x-\dfrac{1(\beta+1)twx^2}
{1-2(tu+v)x-\dfrac{2(\beta+2)twx^2}
{1-3(tu+v)x-\dfrac{3(\beta+3)twx^2}
{\cdots}}}}.
\end{multline}
Define the polynomial $d_{n,k}(\beta,t,u,v)$ to be  the coefficients of $w^k$ in $D_{n}(\beta,t,u,v,w)$:
$$
D_{n}(\beta,t,u,v,w)=\sum_{k=0}^{\floor{(n-1)/2}} d_{n,k}(\beta,t,u,v) w^k.
$$
Substituting $x \leftarrow \frac{x}{(tu+v)}$ and $w \leftarrow \frac{w(tu+v)^2}{t}$ in \eqref{eq:C},  we get
\begin{align}\label{eq:dcf}
\sum_{n\ge 1} \sum_{k=0}^{\floor{(n-1)/2}} \frac{d_{n,k}(\beta,t,u,v)}{t^k(tu+v)^{n-1-2k}} w^k x^{n-1}=
\dfrac{1}
{1-x-\dfrac{(\beta+1)wx^2}
{1-2x-\dfrac{2(\beta+2)wx^2}
{1-3x-\dfrac{3(\beta+3)wx^2}
{\cdots}}}}.
\end{align}
So $d_{n,k}^*(\beta):=\frac{d_{n,k}(\beta,t,u,v)}{t^k(tu+v)^{n-1-2k}}$ must be a polynomial   in $\beta$, since the right-hand side of \eqref{eq:dcf} is independent of $u$, $v$ and $t$. Hence we have
\begin{align}\label{eq:starD}
D_{n}(\beta,t,u,v,w)
=\sum_{k=0}^{\floor{(n-1)/2}} d_{n,k}^*(\beta) t^k(tu+v)^{n-1-2k} w^k.
\end{align}
To verify \eqref{def:c}, taking $t=u=1$ and $v=0$ in \eqref{eq:defD} and \eqref{eq:starD}  we obtain
$$
D_n(\beta,1,1,0,w)=\sum_{k=0}^{\floor{(n-1)/2}} d_{n,k}(\beta) w^k=\sum_{k=0}^{\floor{(n-1)/2}} d_{n,k}^*(\beta) w^k.
$$
Thus $d_{n,k}(\beta)=d_{n,k}^*(\beta)$,
 and we get \eqref{eq:c} by putting $u=v=w=1$ in \eqref{eq:starD} .
Clearly,  for all $k\ge1$,  the polynomial $d_{n,k}(\beta)$ has a factor $(\beta+1)$ because of the presence of
the term  $(\beta+1)w$ at the second row of the continued fraction on the right side of \eqref{eq:dcf}.
\end{proof}
\section{Concluding remark}\label{sec:conclusion}
Consider the descent polynomial of involutions on $[n]$:
\begin{align}
I_{n}(t):=\sum_{\sigma \in \I_n} t^{\des\sigma}=I(n,0)+I(n,1)t+\cdots + I(n, n-1)t^{n-1},
\end{align}
where $\I_n$ is  the subset of involutions in  $\S_{n}$.
The sequence $\{I(n,k)\}_{0\leq k\leq n-1}$
  is known  \cite{GZ06} to be \emph{symmetric} and \emph{unimodal}.
F. Brenti had conjectured that the sequence $\{I(n,k)\}_{0\leq k\leq n-1}$ is \emph{log-concave}, which was later  disproved by Barnabei et al.\cite{BFS09}. Hence the best result one could expect  for $I_n(t)$ is the following expansion
\begin{align}
I_{n}(t)=\sum_{k=0}^{\lfloor (n-1)/2\rfloor}\alpha_{n,k}t^{k}(1+t)^{n-1-2k} \qquad \textrm{with}\quad \alpha_{n,k}\in \N.
\end{align}
This is Conjecture~4.1   in \cite{GZ06}.
Unfortunately, the generating function of $I_{n}(t)$  does not have  a nice continued fraction expansion as that for $A_n(t)$ or $B_n(t)$.

\section*{Acknowledgement}
This work was  supported by the French National Research Agency  under
 the grant ANR-08-BLAN-0243-03 and the program MIRA Recherche 2008 (project 08 034147 01) de
la  R\'egion Rh\^one-Alpes.


\section*{Appendix}
The first values of $a_{n,k}(p,q)$ are given by $a_{n,0}(p,q)=1$ for $1\leq n\leq 5$ and
$${\small
\begin{tabular}{c|ccc}
  n $\setminus$ k& 0& 1 & 2  \\
  \hline
  1 & 1 &       \\
  2 & 1 &       \\
  3 & 1 & $p+q$ &    \\
  4 & 1 & $(p+q)(p+q+2)$ \\
  5 & 1 & $(p+q)[(p+q)^2+2(p+q)+3]$ & $(p+q)^2(p^2+pq+q^2+1)$ \\
\end{tabular}
}$$

For $j=0,1,2$ the first values of the polynomials $b_{n,k,j}(p,q)$ are given as follows.
$${\small
\begin{tabular}{r|rccc}
j=0& k=0 & 1 & 2 & 3 \\
 \hline
n= 1 & 0 &        \\
 2 & 0 & 1 &    \\
 3 & 0 & 1 &    \\
 4 & 0 & 1 & $(p+q)^2+1$ \\
 5 & 0 & 1 & $(p+q)^3+2(p+q)^2+2$ \\
 6 & 0 & 1 & $(p+q)^4+2(p+q)^3+3(p+q)^2+3$ & $b_{6,3,0}(p,q)$\\
\end{tabular}
}$$
$${\small
\begin{tabular}{r|rcc}
j=1& k=0 & 1 & 2  \\
 \hline
 n=1 & 1 &        \\
 2 & 0 &        \\
 3 & 0 & $p+2$ &    \\
 4 & 0 & $2p+2$ &   \\
 5 & 0 & $3p+2$ & $b_{5,2,1}(p,q)$ \\
\end{tabular}
\hspace{2cm}
\begin{tabular}{r|rcc}
j=2& k=0 & 1 & 2 \\
 \hline
n= 2 & 1 &     \\
 3 & 0 &     \\
 4 & 0 & $p^2+2p+3$ \\
 5 & 0 & $3p^2+4p+3$  \\
 6 & 0 & $6p^2+6p+3$ & $b_{6,2,2}(p,q)$ \\
\end{tabular}
}$$
where
\begin{align*}{\scriptsize
b_{6,3,0}(p,q)&=(p+q)^6+(1-2pq)(p+q)^4+(2+p^2q^2)(p+q)^2+1,\\
b_{5,2,1}(p,q)&=(p^2+2p+2)q^2 + (2p^3+4p^2+4p)q + (p^4+2p^3+2p^2+2p+3),\\
b_{6,2,2}(p,q)&=(p^4+2p^3+5p^2+4p+3)q^2 + (2p^5+4p^4+10p^3+8p^2+6p)q \\
&\qquad + (p^6+2p^5+5p^4+4p^3+6p^2+6p+6).
}\end{align*}

The first non-zero values of $c_{n,k}(\beta)$ ($2\leq n\leq 7$) are given by the following table.
$${\small
\begin{tabular}{c|cccc}
  n $\setminus$ k& 1 & 2 & 3 \\
  \hline
  2  & $\beta$ &    \\
  3  & $\beta$ &    \\
  4  & $\beta$ & $\beta(3\beta+2)$ \\
  5  & $\beta$ & $2\beta(5\beta+4)$ \\
  6  & $\beta$ & $\beta(25\beta+22)$ & $ \beta(15\beta^2 + 30\beta + 16)$\\
  7  & $\beta$ & $4\beta(14\beta+13)$ & $ \beta(105\beta^2 + 238\beta + 136)$\\
\end{tabular}
}$$

The first values of $d_{n,k}(\beta)$ are given by $d_{n,0}(\beta)=1$ for $1\leq n\leq 7$ and
$${\small
\begin{tabular}{c|cccc}
  n $\setminus$ k& 0& 1 & 2 & 3 \\
  \hline
  1 & 1 &       \\
  2 & 1 &       \\
  3 & 1 & $\beta+1$ &    \\
  4 & 1 & $4(\beta+1)$ \\
  5 & 1 & $11(\beta+1)$ & $(\beta+1)(3\beta+5)$ \\
  6 & 1 & $26(\beta+1)$ & $(\beta+1)(25\beta+43)$ \\
  7 & 1 & $57(\beta+1)$ & $10(\beta+1)(13\beta+23)$ & $(\beta+1)(15\beta^2+60\beta+61)$\\
\end{tabular}.
}$$

%

\providecommand{\bysame}{\leavevmode\hbox to3em{\hrulefill}\thinspace}
\providecommand{\href}[2]{#2}


\end{document}